\documentclass[11pt,a4paper]{article}
\usepackage{amsthm,amssymb,amsfonts,amsmath,lineno,mathtools,graphicx}

\usepackage{epsfig,xspace,algorithm,verbatim}
\usepackage[noend]{algpseudocode}
\usepackage{color}
\usepackage{enumerate}
\usepackage{tikz}
\usepackage{pgfmath} 
\usepackage[b]{esvect}
\usepackage{comment}
\usetikzlibrary{decorations.markings}

\newtheorem{theorem}{Theorem}

\newtheorem{property}{Property}

\begin{document}
\tikzset{->-/.style={decoration={
  markings,
  mark=at position .5 with {\arrow{>}}},postaction={decorate}}}

\title{The Oriented Chromatic Number \\ of the Hexagonal Grid is 6}

\author{Antoni Lozano\thanks{CS Dept., Universitat Polit\`ecnica de Catalunya, 08034-Barcelona, Catalonia, {\tt antoni.lozano@upc.edu}}}

\date{}

\maketitle


\begin{abstract}\noindent
The oriented chromatic number of a directed graph $G$ is the minimum order of an oriented graph to which $G$ has a homomorphism. The oriented chromatic number $\chi_o({\cal F})$ of a graph family ${\cal F}$ is the maximum oriented chromatic number over any orientation of any graph in ${\cal F}$. For the family of hexagonal grids ${\cal H}_2$, Bielak (2006) proved that $5 \le \chi_o({\cal H}_2) \le 6$. Here we close the gap by showing that $\chi_o({\cal H}_2) \ge 6$.

\smallskip\smallskip
{\small\noindent
{\bf Keywords}: graph, oriented coloring, hexagonal grid.\\
\noindent
{\bf 2020 Mathematics Subject Classification}: 05C15, 05C85. }
\end{abstract}

\section{Introduction}\label{sec:int}

An {\em oriented graph} is a digraph $\vv{G}$ obtained from a simple graph $G$ by assigning to each edge one of its two possible directions. Digraph $\vv{G}$ is then called {\em an orientation of} $G$. A {\em tournament} ({\em $k$-tournament}) is an orientation of the complete graph ($K_k$).

All graphs considered here are either finite simple graphs or their orientations. As usual, $V(G)$ denotes the set of vertices of an unoriented graph $G$, $E(G)$ denotes its set of edges, and $A(\vv{G})$ denotes the arc set of the orientation $\vv{G}$ of $G$. For any arc $(u,v)$, we say that $u$ dominates $v$. The {\em outdegree} of $u$ in $\vv{G}$, denoted by ${\rm deg}_{\vv{G}}^+(u)$ is the number of vertices dominated by $u$. The {\em score set} of a tournament $T$ is the set $\{ {\rm deg}_T^+(u) \}_{u \in V(T)}$.

\subsection{Oriented chromatic number}

An {\em oriented $k$-coloring} of an oriented graph is a homomorphism to an orientation of $K_k$. Therefore, an oriented $k$-coloring of an oriented graph $\vv{G}$ is a function $\phi: V(\vv{G}) \rightarrow V(\vv{K_k}$) such that if $(u,v) \in A(\vv{G})$, then $(\phi(u),\phi(v)) \in \vv{K}_k$, where $\vv{K}_k$ is an orientation of $K_k$. For any vertex $u \in V(\vv{G})$, we say that $\phi(u)$ is its {\em color} in the oriented $k$-coloring $\phi$ of $\vv{G}$. A common characterization of an oriented coloring $\phi$ of $\vv{G}$ is as a (proper) $k$-coloring for which colors can only be assigned to vertices in a given direction, that is, if $(u,v), (s,t) \in A(\vv{G})$ and $\phi(u) = \phi(t)$, then $\phi(v) \ne \phi(s)$.

The {\em oriented chromatic number $\chi_o(\vv{G})$ of an oriented graph $\vv{G}$} is the minimum number $k$ for which $\vv{G}$ has an oriented $k$-coloring.
The {\em oriented chromatic number $\chi_o(G)$ of an unoriented graph $G$} is defined as the maximum $\chi_o(\vv{G})$ over all orientations $\vv{G}$ of $G$. If ${\cal F}$ is a family of graphs, the oriented chromatic number of ${\cal F}$ is defined as $\chi_o({\cal F}) = \max \{ \chi_o(G) \mid G \in {\cal F} \}$. See \cite{S16} for a survey of the main results regarding the oriented chromatic number and \cite{BDLNS} for the related concept of pushable chromatic number applied to grids.

Some bounds on the oriented chromatic number are known for several graph families like planar graphs \cite{RS}, outerplanar graphs \cite{S97}, graphs with bounded degree \cite{S97,KSZ,W07}, Halin graphs \cite{DS}, Cartesian products \cite{S12,W05}, hypergraphs \cite{W07}, or square grids \cite{FRR,DN,B}.

\subsection{Grids}

The {\em square grid} $S(m,n)$ is defined as the Cartesian product $P_m \square P_n$. The vertices of $S(m,n)$ will be represented by $v_{i,j}$, for $1 \le i \le m$ and $1 \le j \le n$, with its edges being all $\{ v_{i,j}, v_{i,j+1} \}$ and $\{ v_{i,j}, v_{i+1,j} \}$, for $v_{i,j}, v_{i,j+1}, v_{i+1,j} \in V(S(m,n))$. The set of all subgraphs of $S(m,n)$, for all $m$ and $n$, is referred to as {\em the square grid} and is denoted by ${\cal S}_2$. Finding the exact value for the oriented chromatic number of the square grid seems to be a surprisingly hard problem. Fertin, Raspaud, and Roychowdhury \cite{FRR} proved that $7 \le \chi_o({\cal S}_2) \le 11$, the lower bound being subsequently improved to 8 by Dybizba\'nski and Nenca \cite{DN}.

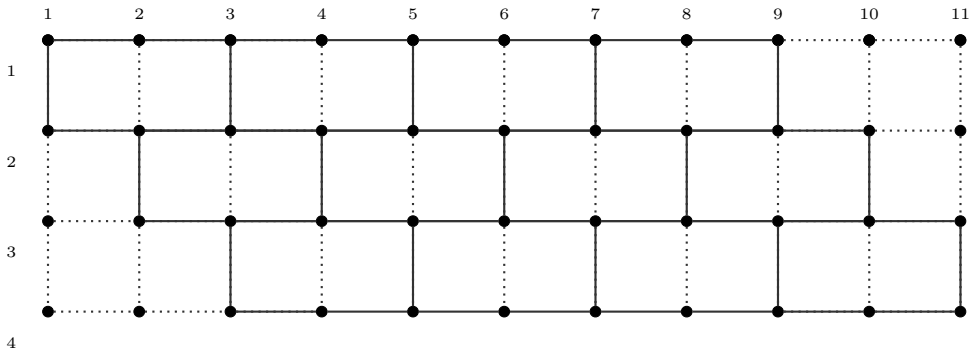
\begin{figure}[H]
\centerline{
\begin{tikzpicture}
[scale=0.55,edge_style/.style={black!80,-,thick}]
\tikzstyle{vertex}=[scale=0.5,circle,draw,minimum size=8pt,inner sep=2pt,fill=black]
\foreach \x in {1,2,3,4}
\node[vertex,label={[left=8pt,below=0.5pt]{\tiny \x}}] (\x-1) at (0,10-2*\x) {};
\foreach \y in {1,2,3,4,5,6,7,8,9,10,11}
\node[vertex,label={[above=2pt]{\tiny \y}}] (1-\y) at (2*\y-2,8) {};
\foreach \x in {1,2,3,4}
\foreach \y in {2,3,4,5,6,7,8,9,10,11}
\node[vertex] (\x-\y) at (2*\y-2,10-2*\x) {};
\foreach \y in {1,3,5,7,9}
\path[edge_style] (1-\y) edge (2-\y);
\foreach \y in {2,4,6,8,10}
\path[edge_style] (2-\y) edge (3-\y);
\foreach \y in {3,5,7,9,11}
\path[edge_style] (3-\y) edge (4-\y);
\foreach \x in {1,2} {
\path[edge_style] (\x-1) edge (\x-2);
\path[edge_style] (\x-2) edge (\x-3);
\path[edge_style] (\x-3) edge (\x-4);
\path[edge_style] (\x-4) edge (\x-5);
\path[edge_style] (\x-5) edge (\x-6);
\path[edge_style] (\x-6) edge (\x-7);
\path[edge_style] (\x-7) edge (\x-8);
\path[edge_style] (\x-8) edge (\x-9);
}
\foreach \x in {2,3} {
\path[edge_style] (\x-2) edge (\x-3);
\path[edge_style] (\x-3) edge (\x-4);
\path[edge_style] (\x-4) edge (\x-5);
\path[edge_style] (\x-5) edge (\x-6);
\path[edge_style] (\x-6) edge (\x-7);
\path[edge_style] (\x-7) edge (\x-8);
\path[edge_style] (\x-8) edge (\x-9);
\path[edge_style] (\x-9) edge (\x-10);
}
\foreach \x in {3,4} {
\path[edge_style] (\x-3) edge (\x-4);
\path[edge_style] (\x-4) edge (\x-5);
\path[edge_style] (\x-5) edge (\x-6);
\path[edge_style] (\x-6) edge (\x-7);
\path[edge_style] (\x-7) edge (\x-8);
\path[edge_style] (\x-8) edge (\x-9);
\path[edge_style] (\x-9) edge (\x-10);
\path[edge_style] (\x-10) edge (\x-11);
}
\foreach \y in {1,2,3,4,5,6,7,8,9,10,11}{
\path[edge_style] (1-\y) edge[dotted] (2-\y);
\path[edge_style] (2-\y) edge[dotted] (3-\y);
\path[edge_style] (3-\y) edge[dotted] (4-\y);}
\foreach \x in {1,2,3,4} {
\path[edge_style] (\x-1) edge[dotted] (\x-2);
\path[edge_style] (\x-2) edge[dotted] (\x-3);
\path[edge_style] (\x-3) edge[dotted] (\x-4);
\path[edge_style] (\x-9) edge[dotted] (\x-10);
\path[edge_style] (\x-10) edge[dotted] (\x-11);}
\end{tikzpicture}}
\caption{The square grid $S(4,11)$ is represented by solid and dotted edges. Solid edges alone represent the hexagonal grid $H_{3,4}$.}
\label{fig:gmn}
\end{figure}

\begin{figure}[H]
\centerline{
\begin{tikzpicture}[scale=0.5,rotate=-30,edge_style/.style={thick,black!80}]
\tikzstyle{vertex}=[scale=0.7,circle,draw,minimum size=8pt,inner sep=3pt]
\node[vertex] (3-4) at ({1.732*(3+4/2)},{3*4/2}) {$v_{1,1}$};
\node[vertex] (3-5) at ({1.732*(3+5/2)},{3*5/2}) {$v_{1,2}$};
\node[vertex] (4-5) at ({1.732*(4+5/2)},{3*5/2}) {$v_{1,3}$};
\node[vertex] (4-6) at ({1.732*(4+6/2)},{3*6/2}) {$v_{1,4}$};
\node[vertex] (5-6) at ({1.732*(5+6/2)},{3*6/2}) {$v_{1,5}$};
\node[vertex] (5-7) at ({1.732*(5+7/2)},{3*7/2}) {$v_{1,6}$};
\node[vertex] (6-7) at ({1.732*(6+7/2)},{3*7/2}) {$v_{1,7}$};
\node[vertex] (6-8) at ({1.732*(6+8/2)},{3*8/2}) {$v_{1,8}$};
\node[vertex] (7-8) at ({1.732*(7+8/2)},{3*8/2}) {$v_{1,9}$};
\node[vertex] (4-3) at ({1.732*(4+3/2)},{3*3/2}) {$v_{2,1}$};
\node[vertex] (5-3) at ({1.732*(5+3/2)},{3*3/2}) {$v_{2,2}$};
\node[vertex] (5-4) at ({1.732*(5+4/2)},{3*4/2}) {$v_{2,3}$};
\node[vertex] (6-4) at ({1.732*(6+4/2)},{3*4/2}) {$v_{2,4}$};
\node[vertex] (6-5) at ({1.732*(6+5/2)},{3*5/2}) {$v_{2,5}$};
\node[vertex] (7-5) at ({1.732*(7+5/2)},{3*5/2}) {$v_{2,6}$};
\node[vertex] (7-6) at ({1.732*(7+6/2)},{3*6/2}) {$v_{2,7}$};
\node[vertex] (8-6) at ({1.732*(8+6/2)},{3*6/2}) {$v_{2,8}$};
\node[vertex] (8-7) at ({1.732*(8+7/2)},{3*7/2}) {$v_{2,9}$};
\node[vertex,inner sep=1pt] (9-7) at ({1.732*(9+7/2)},{3*7/2}) {$v_{2,10}$};
\node[vertex] (6-2) at ({1.732*(6+2/2)},{3*2/2}) {$v_{3,2}$};
\node[vertex] (7-2) at ({1.732*(7+2/2)},{3*2/2}) {$v_{3,3}$};
\node[vertex] (7-3) at ({1.732*(7+3/2)},{3*3/2}) {$v_{3,4}$};
\node[vertex] (8-3) at ({1.732*(8+3/2)},{3*3/2}) {$v_{3,5}$};
\node[vertex] (8-4) at ({1.732*(8+4/2)},{3*4/2}) {$v_{3,6}$};
\node[vertex] (9-4) at ({1.732*(9+4/2)},{3*4/2}) {$v_{3,7}$};
\node[vertex] (9-5) at ({1.732*(9+5/2)},{3*5/2}) {$v_{3,8}$};
\node[vertex] (10-5) at ({1.732*(10+5/2)},{3*5/2}) {$v_{3,9}$};
\node[vertex,inner sep=1pt] (10-6) at ({1.732*(10+6/2)},{3*6/2}) {$v_{3,10}$};
\node[vertex,inner sep=1pt] (11-6) at ({1.732*(11+6/2)},{3*6/2}) {$v_{3,11}$};
\node[vertex] (8-1) at ({1.732*(8+1/2)},{3*1/2}) {$v_{4,3}$};
\node[vertex] (9-1) at ({1.732*(9+1/2)},{3*1/2}) {$v_{4,4}$};
\node[vertex] (9-2) at ({1.732*(9+2/2)},{3*2/2}) {$v_{4,5}$};
\node[vertex] (10-2) at ({1.732*(10+2/2)},{3*2/2}) {$v_{4,6}$};
\node[vertex] (10-3) at ({1.732*(10+3/2)},{3*3/2}) {$v_{4,7}$};
\node[vertex] (11-3) at ({1.732*(11+3/2)},{3*3/2}) {$v_{4,8}$};
\node[vertex] (11-4) at ({1.732*(11+4/2)},{3*4/2}) {$v_{4,9}$};
\node[vertex,inner sep=1pt] (12-4) at ({1.732*(12+4/2)},{3*4/2}) {$v_{4,10}$};
\node[vertex,inner sep=1pt] (12-5) at ({1.732*(12+5/2)},{3*5/2}) {$v_{4,11}$};
\foreach \s/\t in {4-3/3-4,3-4/3-5,4-5/3-5,5-4/4-5,4-6,4-5,6-5/5-6,5-6/4-6,6-4/6-5,4-5/4-6,5-3/4-3,5-3/5-4,5-4/6-4,5-6/5-7,5-7/6-7,6-7/7-6,7-6/7-5,6-5/7-5,6-7/6-8,6-8/7-8,7-8/8-7,8-7/8-6,8-6/7-6,6-2/5-3,6-2/7-2,7-2/7-3,7-3/6-4,7-3/8-3,8-3/8-4,8-4/7-5,8-4/9-4,9-4/9-5,9-5/8-6,9-5/10-5,10-5/10-6,10-6/9-7,9-7/8-7,8-1/7-2,8-1/9-1,9-1/9-2,9-2/8-3,9-2/10-2,10-2/10-3,10-3/9-4,10-3/11-3,11-3/11-4,11-4/10-5,11-4/12-4,12-4/12-5,12-5/11-6,11-6/10-6}
\path[edge_style] (\s) edge [-] (\t);
\end{tikzpicture}}
\caption{Hexagonal grid $H_{3,4}$ with vertex names.
}
\label{fig:hmn}
\end{figure}
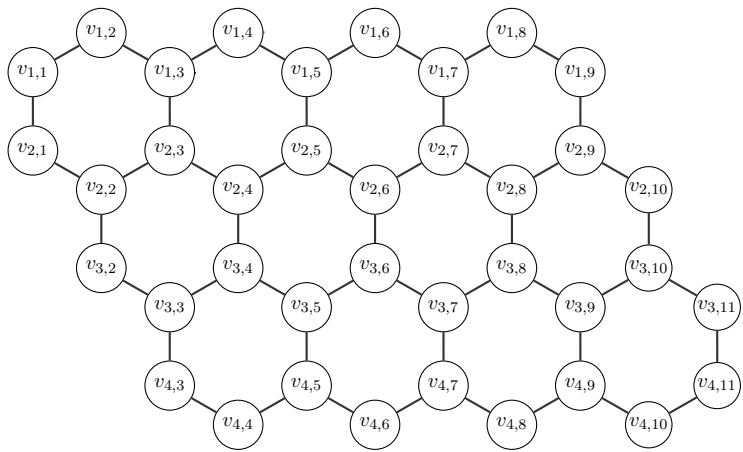

The {\em hexagonal grid} $H_{m,n}$ is defined in \cite{B} (only for $m \le 2$) as $m$ rows of $n$ hexagons (see Figure~\ref{fig:hmn}) and corresponds to the  subgraph of $S(m+1,2n+m)$ just containing the vertices $v_{i,j}$ such satisfy $i-1 \le j \le i+2n$, all edges of the kind $\{ v_{i,j}, v_{i,j+1} \}$, and edges of the kind $\{v_{i,j},v_{i+1,j}\}$ for even $i+j$ (see Figure~\ref{fig:gmn}). The set of all subgraphs of $H_{m,n}$, for all $m$ and $n$, is referred to as {\em the hexagonal grid} and is denoted by ${\cal H}_2$.

Bielak \cite{B} proved that $\chi_o(H_{1,n}) = \chi_o(H_{2,n}) = 5$ for any $n > 3$  (the case of $\chi_o(H_{2,n})$ is a consequence of her proof for fat hexagonal trees). For the general case of the hexagonal grid, she also proved that $5 \le \chi_o({\cal H}_2) \le 6$. In Section~\ref{section:upper}, we give an alternative proof of Bielak's upper bound, while in Section~\ref{section:lower} we show that 6 is actually a lower bound for $\chi_o({\cal H}_2)$ by means of two counterexamples checked by computer.

\section{Upper bound}\label{section:upper}

Bielak \cite{B} showed that the hexagonal grid has oriented chromatic number at most 6. For completeness, we include a proof based on the ideas used by Fertin {\em et al.} \cite{FRR} to show that the square grid has oriented chromatic number at most 11.

\begin{theorem}
$\chi_o({\cal H}_2) \le 6$.
\end{theorem}
\begin{proof}
Consider the orientation $A_6$ of $K_6$ depicted in Figure~\ref{fig:a6} (already used in \cite{B,FRR}). We are interested in the following property of $A_6$, which has been checked by computer (and can be reasonably checked by hand as well).

\begin{property}\label{prop:a6}
For any two vertices $u$ and $v$, there exists a path between $u$ and $v$ of unoriented length 3, whatever the orientations of the arcs on the path. 
\end{property}

\begin{figure}[H]
\centerline{
\begin{tikzpicture}[scale=1.2,rotate=-30,edge_style/.style={thick,black!80}]
\tikzstyle{vertex}=[scale=0.8,circle,draw,minimum size=10pt,inner sep=2pt]
\node[vertex] (0) at ({1.732*(1+6/2)},{3*6/2}) {0};
\node[vertex] (1) at ({1.732*(2+6/2)},{3*6/2}) {1};
\node[vertex] (2) at ({1.732*(3+5/2)},{3*5/2}) {2};
\node[vertex] (3) at ({1.732*(3+4/2)},{3*4/2}) {3};
\node[vertex] (4) at ({1.732*(2+4/2)},{3*4/2}) {4};
\node[vertex] (5) at ({1.732*(1+5/2)},{3*5/2}) {5};
\foreach \s/\t in {0/1,0/2,0/3,1/2,1/5,2/3,2/4,2/5,3/4,3/1,4/5,4/0,4/1,5/0,5/3}
\path[edge_style] (\s) edge [->-,>=stealth] (\t);
\end{tikzpicture}}
\caption{Digraph $A_6$}
\label{fig:a6}
\end{figure}
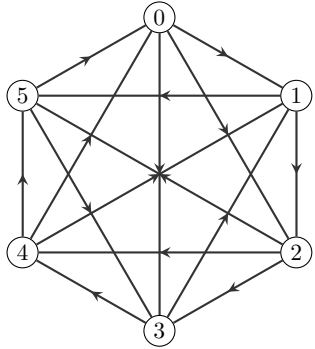

Now, the following is an algorithm that colors any given orientation $\vv{H}_{m,n}$ of $H_{m,n}$ by a homomorphism to $A_6$. Consider the following steps keeping in mind Figure~\ref{fig:hmn}.
\begin{enumerate}
\item {\bf First row.} The first row $v_{1,1}, v_{1,2}, \dots, v_{1,2n+1}$ of vertices in $\vv{H}_{m,n}$ can be clearly colored by a homomorphism to $A_6$. We start with any color for $v_{1,1}$. Then, since both the indegree and the outdegree of any vertex in $A_6$ are at least 2, a new color is always available for the succeeding vertex. 
\item {\bf Following rows.} Suppose $k-1$ rows have already been colored. In order to color the $k$-th row, we color the first vertex $v_{k,k-1}$ in the row with a color compatible with that of $v_{k-1,k-1}$, its neighbor in row $k-1$, which is already colored. Now, if $u$ and $v$ are the colors of $v_{k,k-1}$ and $v_{k-1,k+1}$, resp., according to Property~\ref{prop:a6} there is a path of unoriented length 3 from $u$ to $v$ in $A_6$ that respects the orientations of the path from $v_{k,k-1}$ to $v_{k-1,k+1}$ in $\vv{H}_{m,n}$. Therefore, the vertices $v_{k,k}$ and $v_{k,k+1}$ can be colored. Taking $v_{k,k+1}$ and $v_{k-1,k+3}$, we can continue with the coloring in the same way until we complete the $k$-th row.
\end{enumerate}

Now, any orientation $\vv{G}$ of an hexagonal grid is a subgraph of some orientation $\vv{H}_{m,n}$ of $H_{m,n}$ for some integers $m$ and $n$. By the above algorithm, $\vv{H}_{m,n}$ has a homomorphism to $A_6$ and, then, so does $\vv{G}$. Therefore, the oriented chromatic number of any hexagonal grid is at most 6.
\end{proof}

\section{Lower bound}\label{section:lower}

Our lower bound is based on the construction of a digraph that cannot be colored by any tournament of order 5. The complete list of 5-tournaments is therefore needed in order to get our result.

According to Theorem 41 in \cite{M} (based on a formula from \cite{D}), there are only 12 tournaments of order 5 which, expressed as the upper triangle of their adjacency matrices in row order (see the list of tournaments of 5 vertices at \cite{K}), are the following:
\begin{align*}
T_1:&\,\,\, {\tt 0000000000} \,\,\,\,\,\,\,\,\,& T_7:&\,\,\, {\tt 1000001000} \\
T_2:&\,\,\, {\tt 0000000101} \,\,\,\,\,\,\,\,\,& T_8:&\,\,\, {\tt 1000010000} \\
T_3:&\,\,\, {\tt 0000100010} \,\,\,\,\,\,\,\,\,& T_9:&\,\,\, {\tt 1000100000} \\
T_4:&\,\,\, {\tt 0000100100} \,\,\,\,\,\,\,\,\,& T_{10}:&\,\,\, {\tt 1000100101}\\
T_5:&\,\,\, {\tt 0001100100} \,\,\,\,\,\,\,\,\,& T_{11}:&\,\,\, {\tt 1000110101}\\
T_6:&\,\,\, {\tt 0010100101} \,\,\,\,\,\,\,\,\,& T_{12}:&\,\,\, {\tt 1100101110} 
\end{align*}
Although the method to generate the above tournaments is not published, it is the same as in \cite{K2} (B. D. McKey, personal communication, December 2019), meaning that they are pairwise nonisomorphic. For a direct check of nonisomorphism, a graph invariant as the score set is not fine enough to give different results on all twelve tournaments, but the following variation is. Define the {\em double score set} of a tournament $T$, written ${\rm ds}(T)$, as the set $\{ \sum_{(u,v) \in A(T)} {\rm deg}_{T}^+(v) \}_{u \in V(T)}$. It is easy to check that the tournaments above have pairwise distinct double score sets, that is, ${\rm ds}(T_i) \ne {\rm ds}(T_j)$ for $i \ne j$, $1 \le i, j \le 12$. Since the double score set is clearly a graph invariant, the previous list must contain all tournaments of order 5.

\smallskip
We base the construction of our digraph on two digraphs that, combined, cannot be colored with any of the twelve tournaments of order 5. In particular, digraph $D_4$ (see Figure~\ref{fig:h6}) is not colorable with tournament $T_5$ and digraph $D_{49}$ (see Figure~\ref{fig:h49}) is not colorable with any other tournament of order 5. Graphs smaller than $D_{49}$ can show non-colorability with most 5-tournaments; a directed 6-cycle, for exemple, is clearly non-colorable with $T_1$, the transitive tournament of order 5. However, $D_{49}$ was specifically designed to act as a counterexample of colorability for $T_{11}$ and, incidentally, constitutes a counterexample for all other 5-tournaments  except $T_5$, thus reducing the number of graphs that are needed for the proof. Graph $D_{49}$ was found by starting with a central directed 6-cycle and then attaching new oriented hexagons to it so that the number of possible $T_{11}$ colorings is as limited as possible. To this end, a path-checking script was used: in concrete, we found it useful that our script produced all possible oriented paths in $T_{11}$ for given source and target colors and a concrete orientation of a path of given length.

\goodbreak

\begin{theorem}
$\chi_o({\cal H}_2) \ge 6$.
\end{theorem}

\begin{proof}
Let $H$ be an undirected graph having an orientation $\vv{H}$ which contains digraphs $D_4$ (in Figure~\ref{fig:h6}) and $D_{49}$ (in Figure~\ref{fig:h49}).

Graph $\vv{H}$ cannot have an oriented 5-coloring since, for each $i$ with $1 \le i \le 12$, there is no homomorphism from either $D_4$ or $D_{49}$ to $T_i$. Although this fact was initially checked by hand (using small graphs for each $i \le 10$), here we use linear programming techniques as a practical tool for checking the absence of homomorphisms.

Our linear program checks whether a given oriented graph $G$ has a homomorphism to tournament $T_i$, $1 \le i \le 12$. Our goal is to know whether there is a homomorphism from $G$ to $T_i$, that is, whether the set of constraints
\[ ({\rm phi}(u),{\rm phi}(v)) \in A(T_i) \mbox{ for all } (u,v) \in A(G) \]
has a solution, where phi$(u)$ is a decision variable for each $u \in V(G)$ and the objective function is irrelevant. An implementation of the previous linear program can be found in Section~\ref{sec:opl}. Using OPL, no solution was found when $G = D_4$ and $i = 5$ and, consequently, we conclude that $D_4$ is not colorable with tournament $T_5$. Similarly, choosing $G = D_{49}$ and any $i \ne 5$, $1 \le i \le 12$, no solution exists neither, meaning that $D_{49}$ is not colorable by any tournament of order 5 distinct from $T_5$. Since $\vv{H}$ contains both $D_4$ and $D_{49}$ as subgraphs, and homomorphisms are preserved for subgraphs, $\vv{H}$ cannot have a homomorphism to any tournament of order 5 and the theorem follows.
\end{proof}

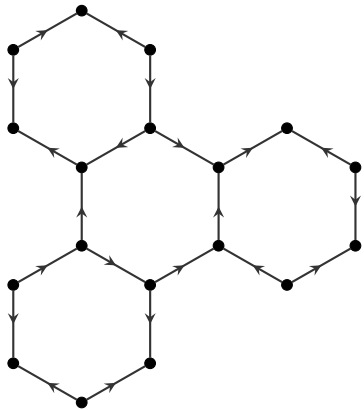
\begin{figure}[H]
\centerline{
\begin{tikzpicture}[scale=0.6,rotate=-30,edge_style/.style={thick,black!80}]
\tikzstyle{vertex}=[scale=0.5,circle,draw,minimum size=18pt,inner sep=3pt]
\foreach \x/\y in {3/4} \node[vertex] (\x\y) at ({1.732*(\x+\y/2)},{3*\y/2}) {6};
\foreach \x/\y in {4/5} \node[vertex] (\x\y) at ({1.732*(\x+\y/2)},{3*\y/2}) {2};
\foreach \x/\y in {3/5} \node[vertex] (\x\y) at ({1.732*(\x+\y/2)},{3*\y/2}) {1};
\foreach \x/\y in {5/4} \node[vertex] (\x\y) at ({1.732*(\x+\y/2)},{3*\y/2}) {3};
\foreach \x/\y in {5/3} \node[vertex] (\x\y) at ({1.732*(\x+\y/2)},{3*\y/2}) {4};
\foreach \x/\y in {4/3} \node[vertex] (\x\y) at ({1.732*(\x+\y/2)},{3*\y/2}) {5};
\foreach \x/\y in {2/4} \node[vertex] (\x\y) at ({1.732*(\x+\y/2)},{3*\y/2}) {7};
\foreach \x/\y in {1/5} \node[vertex] (\x\y) at ({1.732*(\x+\y/2)},{3*\y/2}) {8};
\foreach \x/\y in {1/6} \node[vertex] (\x\y) at ({1.732*(\x+\y/2)},{3*\y/2}) {9};
\foreach \x/\y in {2/6} \node[vertex] (\x\y) at ({1.732*(\x+\y/2)},{3*\y/2}) {\footnotesize 10};
\foreach \x/\y in {4/6} \node[vertex] (\x\y) at ({1.732*(\x+\y/2)},{3*\y/2}) {\footnotesize 11};
\foreach \x/\y in {5/6} \node[vertex] (\x\y) at ({1.732*(\x+\y/2)},{3*\y/2}) {\footnotesize 12};
\foreach \x/\y in {6/5} \node[vertex] (\x\y) at ({1.732*(\x+\y/2)},{3*\y/2}) {\footnotesize 13};
\foreach \x/\y in {6/4} \node[vertex] (\x\y) at ({1.732*(\x+\y/2)},{3*\y/2}) {\footnotesize 14};
\foreach \x/\y in {6/2} \node[vertex] (\x\y) at ({1.732*(\x+\y/2)},{3*\y/2}) {\footnotesize 15};
\foreach \x/\y in {6/1} \node[vertex] (\x\y) at ({1.732*(\x+\y/2)},{3*\y/2}) {\footnotesize 16};
\foreach \x/\y in {5/1} \node[vertex] (\x\y) at ({1.732*(\x+\y/2)},{3*\y/2}) {\footnotesize 17};
\foreach \x/\y in {4/2} \node[vertex] (\x\y) at ({1.732*(\x+\y/2)},{3*\y/2}) {\footnotesize 18};
\foreach \s/\t in {26/35,26/16,15/16,15/24,34/24,35/34,43/34,43/53,53/54,54/45,35/45,64/54,64/65,56/65,56/46,45/46,53/62,61/62,61/51,42/51,42/43}
\path[edge_style] (\s) edge [->-,>=stealth] (\t);
\end{tikzpicture}}
\caption{Digraph $D_4$}
\label{fig:h6}
\end{figure}

\begin{figure}[H]
\centerline{
\begin{tikzpicture}[scale=0.6,rotate=-30,edge_style/.style={thick,black!80}]
\tikzstyle{vertex}=[scale=0.5,circle,draw,minimum size=8pt,inner sep=2pt,fill=black]
\foreach \x/\y in {0/13,0/14,1/8,1/9,1/11,1/12,1/14,1/15,
2/6,2/7,2/9,2/10,2/12,2/13,2/15,
3/5,3/7,3/8,3/10,3/11,3/13,3/14,
4/5,4/6,4/8,4/9,4/11,4/12,4/14,4/15,
5/4,5/6,5/7,5/9,5/10,5/12,5/13,5/15,5/16,
6/4,6/5,6/7,6/8,6/10,6/11,6/13,6/14,6/16,
7/2,7/3,7/5,7/6,7/8,7/9,7/11,7/12,7/14,7/15,
8/1,8/3,8/4,8/6,8/7,8/9,8/10,8/12,8/13,8/15,
9/1,9/2,9/4,9/5,9/7,9/8,9/10,9/11,9/13,9/14,
10/0,10/2,10/3,10/5,10/6,10/8,10/9,10/11,10/12,
11/0,11/1,11/3,11/4,11/6,11/7,11/9,11/10,11/12,
12/1,12/2,12/4,12/5,12/7,12/8,12/10,12/11,
13/2,13/3,13/5,13/6,13/8,13/9,13/11,14/1,14/3,14/4,14/6,14/7,14/9,14/10,
15/1,15/2,15/4,15/5,15/7,15/8,16/2,16/3}
    \node[vertex] (\x-\y) at ({1.732*(\x+\y/2)},{3*\y/2}) {};
\foreach \s/\t in {
8-7/7-8,7-8/7-9,7-9/8-9,8-9/9-8,9-8/9-7,9-7/8-7,
7-9/6-10,8-9/8-10,9-8/10-8,9-7/10-6,8-7/8-6,7-8/6-8,
6-8/5-9,5-9/5-10,5-10/6-10,6-10/6-11,6-11/7-11,7-11/8-10,
8-10/9-10,9-10/10-9,10-9/10-8,10-8/11-7,11-7/11-6,11-6/10-6,
10-6/10-5,10-5/9-5,9-5/8-6,8-6/7-6,7-6/6-7,6-7/6-8,
4-11/4-12,4-12/5-12,5-12/5-13,5-13/6-13,
7-12/6-13,7-12/8-12,8-12/9-11,9-11/10-11,
10-11/11-10,11-9/11-10,11-9/12-8,12-8/12-7,
12-7/13-6,13-6/13-5,12-5/13-5,12-5/12-4,12-4/11-4,
11-4/11-3,11-3/10-3,9-4/10-3,9-4/8-4,8-4/7-5,7-5/6-5,
6-5/5-6,5-7/5-6,5-7/4-8,4-8/4-9,4-9/3-10,3-10/3-11,4-11/3-11,
6-7/5-7,4-9/5-9,5-10/4-11,5-12/6-11,7-11/7-12,9-11/9-10,
10-9/11-9,12-7/11-7,11-6/12-5,11-4/10-5,9-5/9-4,7-5/7-6,
2-10/3-10,3-11/2-12,2-12/1-12,1-11/1-12,2-10/1-11,
2-12/2-13,2-13/3-13,4-12/3-13,3-13/3-14,3-14/4-14,4-14/5-13,
4-14/4-15,4-15/5-15,6-14/5-15,6-13/6-14,
6-14/7-14,7-14/8-13,8-12/8-13,8-13/9-13,9-13/10-12,10-12/10-11,
10-12/11-12,11-12/12-11,12-10/12-11,11-10/12-10,
12-10/13-9,13-9/13-8,12-8/13-8,13-8/14-7,14-7/14-6,14-6/13-6,
14-6/15-5,15-5/15-4,14-4/15-4,13-5/14-4,
14-4/14-3,14-3/13-3,12-4/13-3,13-3/13-2,13-2/12-2,12-2/11-3,
12-2/12-1,12-1/11-1,10-2/11-1,10-3/10-2,
10-2/9-2,9-2/8-3,8-4/8-3,8-3/7-3,7-3/6-4,6-4/6-5,
6-4/5-4,5-4/4-5,4-6/4-5,5-6/4-6,4-6/3-7,3-7/3-8,4-8/3-8,
3-8/2-9,2-9/2-10,
1-12/0-13,0-14/0-13,1-14/0-14,2-13/1-14,
1-14/1-15,1-15/2-15,2-15/3-14,
5-15/5-16,6-16/5-16,7-15/6-16,7-14/7-15,
7-15/8-15,8-15/9-14,9-14/9-13,
12-11/13-11,14-10/13-11,14-9/14-10,13-9/14-9,
14-9/15-8,15-8/15-7,15-7/14-7,
15-4/16-3,16-2/16-3,15-2/16-2,14-3/15-2,
15-2/15-1,15-1/14-1,14-1/13-2,
11-1/11-0,10-0/11-0,9-1/10-0,9-2/9-1,
9-1/8-1,8-1/7-2,7-2/7-3,
4-5/3-5,2-6/3-5,2-7/2-6,3-7/2-7,
2-7/1-8,1-8/1-9,1-9/2-9}
\path[edge_style] (\s) edge [->-,>=stealth] (\t);
\end{tikzpicture}}
\caption{Digraph $D_{49}$}
\label{fig:h49}
\end{figure}
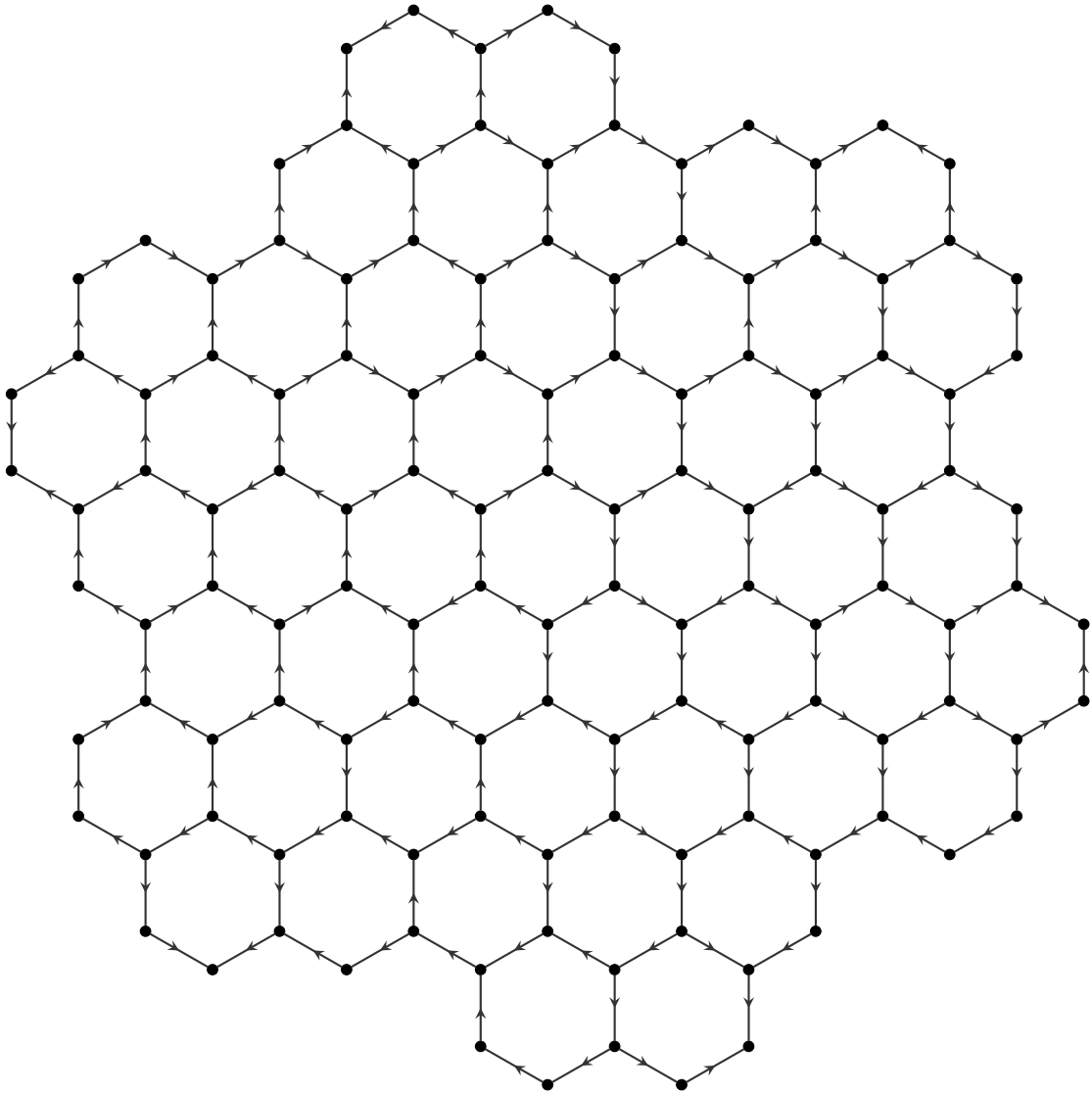

\section{Linear program}
\label{sec:opl}

We code our linear program in OPL and use IBM ILOG CPLEX Optimization Studio to run it. Remember that our goal is to decide whether an oriented graph $G$ has a homomorphism to a 5-tournament $T_i$, $1 \le i \le 12$. The data needed for our model is, then, $G$ and $T$, which in our implementation will be described by the input variables {\tt N} $= |V(G)|$, {\tt M} $= |A(G)|$, {\tt arc} $= A(G)$, and {\tt T} $= T_i$. Assuming that $V(T_i) \subseteq \{0,\dots,4\}$, we code every arc $(u,v)$ of $T_i$ as the positive integer $u+5v$, which clearly identifies the arc. For instance, using this encoding, arc $(0,4)$ of $T_5$ is encoded by number 20 and the whole tournament $T_5$ by {\tt T = [1 2 3 8 9 11 14 17 19 20]}. The model file is the following.

{\small
\begin{verbatim}
  int N = ...;
  int M = ...;
  range I = 1..10;
  range Vertices = 1..N;
  range Arcs = 1..M;
 
  int T[i in I] = ...;
  int arc[Arcs][1..2] = ...;
 
  dvar int phi[j in Vertices] in 0..4;
  dvar boolean z[i in I,a in Arcs];

  minimize 1;
  
  subject to { 
    // the coloring is correct
    forall (a in Arcs) 
        phi[arc[a][1]] + 5*phi[arc[a][2]] == sum(i in I) z[i,a]*T[i];
	
    // there is only one i for which z[i,a] equals 1 for each arc a
    forall (a in Arcs) sum(i in I) z[i,a] == 1; 	
  }
\end{verbatim}}

In our linear program, then, we check whether ({\tt phi[arc[a][1]]}, {\tt phi[arc[a][2]]}) is an arc in {\tt T} using the equation

\begin{verbatim}       
  phi[arc[a][1]] + 5*phi[arc[a][2]] == sum(i in I) z[i,a]*T[i];
\end{verbatim}

\noindent 
Here we introduce new decision variables {\tt z[i,a]} to guarantee that each of the numbers obtained on the left of the equation (corresponding to arc {\tt a}) equals one of the positive integers in the list {\tt T}.

\medskip
As an example, the data file to check that $D_4$ cannot be colored with $T_5$ is the following (vertex numbers in {\tt arc} follow the numbering in Figure~\ref{fig:h6}).

{\small
\begin{verbatim}
  N = 18;
  M = 21;
  T = [1 2 3 8 9 11 14 17 19 20];
  arc = [[1 2] [3 2] [4 3] [5 4] [5 6] [1 6] [6 7] [8 7] [8 9] [10 9] [10 1] 
    [2 11] [12 11] [12 13] [14 13] [14 3] [4 15] [16 15] [16 17] [18 17] [18 5]];
\end{verbatim}}

\section{Acknowledgments}
The authors would like to thank Janusz Dybizba\'nski for patiently answering our first doubts about the topic and Brendan McKay for his useful suggestions. This research was supported by MICIN under project PID2019-104129GB-I00 and by Universitat Polit\`ecnica de Catalunya under grant AGRUPS-2023.

\section{Declaration of competing interest}

The authors declare that they have no known competing financial interests or personal relationships that could have appeared to influence this work.

\goodbreak

\end{document}